\documentclass[a4paper,oneside,11pt, reqno, psamsfonts]{amsart}
\usepackage{etoolbox}
\makeatletter
\patchcmd\maketitle
  {\uppercasenonmath\shorttitle}
  {}
  {}{}
\patchcmd\maketitle
  {\@nx\MakeUppercase{\the\toks@}}
  {\the\toks@}
  {}
  {}{}
\patchcmd\@settitle{\uppercasenonmath\@title}{\Large}{}{}
\patchcmd\@setauthors
  {\MakeUppercase{\authors}}
  {\authors}
  {}{}
\makeatother
\usepackage{amsmath,amssymb,amsthm, amsfonts}
\usepackage{mathrsfs}
\usepackage{color}
\usepackage{url}
\usepackage{tikz-cd}
\usepackage[utf8]{inputenc}
\usepackage[T1]{fontenc}

\newtheorem{theorem}{Theorem}[section]
\newtheorem{question}{Question}[section]
\newtheorem{definition}{Definition}[section]

\newtheorem{corollary}{Corollary}[section]
\newtheorem{proposition}{Proposition}[section]
\newtheorem{lemma}{Lemma}[section]
\newtheorem{remark}{Remark}[section]
\newtheorem{example}{Example}[section]
\numberwithin{equation}{section}
\newcommand{\beq}{\begin{eqnarray}}
\newcommand{\eeq}{\end{eqnarray}}
\newcommand{\beqn}{\begin{eqnarray*}}
\newcommand{\eeqn}{\end{eqnarray*}}
\newcommand{\rar}{\rightarrow}
\newcommand*{\Ge}{\geqslant}
\newcommand*{\Le}{\leqslant}

\usepackage[colorlinks=true]{hyperref}
\hypersetup{urlcolor=blue, citecolor=red , linkcolor= blue}
\usepackage[capitalise,noabbrev,nameinlink]{cleveref}      
\begin{document}
\author[S. Chavan  \and Chaman Kumar Sahu] {\Large{Sameer Chavan} \and \Large{Chaman Kumar Sahu}}
\address{Department of Mathematics and Statistics\\
Indian Institute of Technology Kanpur, India}
\email{\url{chavan@iitk.ac.in}}
\email{\url{chamanks@iitk.ac.in}}

\date{\today}

\title[Dirichlet polynomials and a moment problem]
{Dirichlet polynomials and a moment problem}
\maketitle

\begin{abstract} 
Consider a linear functional $L$ defined on the space $\mathcal D[s]$ of Dirichlet polynomials with real coefficients and the set $\mathcal D_+[s]$ of non-negative elements in $\mathcal D[s].$  An analogue of the Riesz-Haviland theorem in this context asks: What are all $\mathcal D_+[s]$-positive linear functionals $L,$ which are moment functionals? Since the space $\mathcal D[s],$ when considered as a subspace of $C([0, \infty), \mathbb R),$ fails to be an adapted space in the sense of Choquet, the general form of Riesz-Haviland theorem is not applicable in this situation.
In an attempt to answer the forgoing question, we arrive at the notion of a moment sequence, which we call the Hausdorff log-moment sequence. Apart from an analogue of the Riesz-Haviland theorem, we show that any Hausdorff log-moment sequence is a linear combination of $\{1, 0, \ldots, \}$ and $\{f(\log(n)\}_{n \geqslant 1}$ for a completely monotone function $f : [0, \infty) \rar [0, \infty).$ 
Moreover, such an $f$ is uniquely determined by the sequence in question.
\end{abstract}
\section{Introduction}

 Let $\mathbb Z_+$ denote the set of positive integers and let $\mathbb R$ denote the set of real numbers.  Let $X$ be a locally compact Hausdorff space.
A {\it Radon measure} on $X$ is a locally finite, inner regular Borel measure on $X.$ 
The real linear space of continuous functions from $X$ into $\mathbb R$ will be denoted by $C(X, \mathbb R).$   
 For a closed subset $K$ of $X$ and a subspace $\mathcal S$ of $C(X, \mathbb R),$ let $\mathcal S^K_+$ denote the cone of all functions in $\mathcal S$ which are non-negative on $K.$ If $K=X,$ then we use the notation $\mathcal S_{+}$ to denote $\mathcal S^K_{+}.$
 For a subset $\mathcal M$ of $\mathcal S,$ 
a linear functional $L$ on $\mathcal S$ is said to be {\it $\mathcal M$-positive} if $L(q) \geqslant 0$ for all $q \in \mathcal M.$
  For a closed subset $K$ of $X,$ we say that $L$ is a {\it $K$-moment functional} if there exists a positive Radon measure $\nu$ concentrated on $K$ such that 
 \beqn 
L(q) = \int_{K} q(s)\nu(ds), \quad q \in \mathcal S.
 \eeqn
 In this case, we refer to the measure $\nu$ as a {\it representing measure} of $L$ (see \cite{Sc, BCR, Si-1} for the basics of moment theory). 
 
 Let $\mathbb H_0$ denote the right half plane $\{z \in \mathbb C : \Re(z) > 0\}$ in the complex plane $\mathbb C.$ 
 A  {\it Dirichlet polynomial} is a function $q : \overline{\mathbb H}_0 \rar \mathbb C$ given by
\beqn q(s)=\sum_{n=1}^k a_n n^{-s}, \quad s \in \overline{\mathbb H}_0, \eeqn
where $a_n \in \mathbb C$ and $k \in \mathbb Z_+.$ 
 The real linear space of Dirichlet polynomials with real coefficients will be denoted by $\mathcal D[s].$ By the uniqueness of the Dirichlet series, $\mathcal D[s]$ can be embedded into $C([0, \infty), \mathbb R)$ via the map $f \mapsto f|_{[0, \infty)}.$
 Moreover, $\mathcal D[s]$ is a unital sub-algebra of $C([0, \infty), \mathbb R)$
  (the reader is referred to \cite[Chapter~4]{QQ} for algebraic properties of Dirichlet polynomials and Dirichlet series). For a closed subset $K$ of $[0, \infty),$ let $\mathcal D^K_{+}[s]$ denote the set of all  functions in $\mathcal D[s]$ which are non-negative on $K.$ 

For a sequence ${\mathbf w} = \{w_{n}\}_{n\geqslant 1}$ of non-negative real numbers, consider the real linear functional $L_{\mathbf w}: \mathcal D[s] \longrightarrow \mathbb{R}$ defined by setting $L_{\mathbf w}(n^{-s}) = w_{n},$ $n \geqslant 1,$ and extending linearly to $\mathcal D[s].$ 
If $L_{\mathbf w}$ is a $K$-moment functional, then clearly $L_{\mathbf w}$ is $\mathcal D^K_{+}[s]$-positive.
An analogue of Riesz-Haviland theorem asks for the converse: 
 
 \begin{question} \label{Q1}
 If $L_{\mathbf w}$ is a $\mathcal D^K_{+}[s]$-positive linear functional on $\mathcal D[s],$ then whether $L_{\mathbf w}$ is a $K$-moment functional? 
 \end{question}
 
 In case $K$ is bounded, the answer to Question~\ref{Q1} is affirmative and this is a consequence of \cite[Proposition~1.9]{Sc}. In general, the answer to Question~\ref{Q1} is No (see Theorem~\ref{Q1-DS-Coro}). So the following natural question arises:
 \begin{question} \label{Q2}
 What are the sequences $\mathbf w$ for which every $\mathcal D^K_{+}[s]$-positive linear functional $L_{\mathbf w}$ on $\mathcal D[s]$ is a $K$-moment functional?
 \end{question}
 To answer Questions~\ref{Q1} and \ref{Q2}, we must obtain an analogue of Riesz-Haviland theorem (see \cite[Theorem~1.12]{Sc}), which replaces the polynomials by Dirichlet polynomials. One approach to the proof of the Riesz-Haviland theorem, as presented in \cite{Sc}, can be based on the notion of an {\it adapted space} (see~\cite[Definition~1.5]{Sc}). To see whether or not $\mathcal D[s]$ is an adapted space, note that
$
 1 \in \mathcal D[s]~\text{and}~\mathcal D[s] = \mathcal D_+[s] - \mathcal D_+[s].$ 
In fact, since $\mathcal D[s]$ is an algebra containing $1,$ $$f(s) =\frac{1}{4}((f(s)+1)^2 - (f(s)-1)^2), \quad s \in [0, \infty)$$ is a difference of two Dirichlet series in $\mathcal D_{+}[s]$. However, $\mathcal D[s]$  does not have the crucial property of the existence of the dominating function (see \cite[Definition~1.5(iii)]{Sc}. Indeed, for $f(s)=1,$  there exists no $g \in \mathcal D_{+}[s]$ with the following property: For any $\epsilon > 0,$ there
exists a compact subset $K_\epsilon$ of $[0, \infty)$ such that $|f(s)| \Le \epsilon |g(s)|$ for all $s \in [0, \infty) \backslash K_\epsilon.$ It is interesting to note that there exist an $E_+$-positive linear functional $L$ on a linear space $E$ (without the aforementioned property), which is not a moment functional (see~\cite[Example~1.11]{Sc} and Corollary~\ref{rmk-Q1}).

\section{Hausdorff log-moment sequences}

One may answer Question \ref{Q1} by combining \cite[Proposition~1.9]{Sc} with a compactification technique (see \cite[Chapter~9]{Sc}). Since this proof is not relevant to the investigations here, we have relegated it to the appendix. Our ploy is to consider a $[0, 1]$-moment problem, which, after a change of variables (see Lemma~\ref{CoV}), yields a more general $[0, \infty)$-moment problem than the one discussed in Section~1.  The solution to the $[0, 1]$-moment problem (which is a simple consequence of \cite[Proposition~1.9]{Sc}) allows us to answer Question \ref{Q1} (see First proof of Theorem~\ref{Q1-DS-Coro}). 
Note that a particular case of the $[0, \infty)$-moment problem appears in \cite[Equation~(1.3)]{Mc-1} in the study of the multiplier algebra of certain Hilbert spaces of the Dirichlet series (see Example~\ref{Exm-DMS-1}(a)).
 This problem also appears in the study of bounded Helson matrices (see \cite{PP}). Indeed, it is easily seen from the discussion prior to \cite[Theorem~5.1]{PP} that if $\nu$ is a representing measure of $L_{\mathbf w}$ satisfying $\nu(\{0\})=0,$ then 
 \beqn
 L_{\mathbf w}(n^{-s}) = \int_{(0, 1)^{\infty}} t^{k(n)} \mathscr B_*\nu(dt), \quad n \Ge 1,
 \eeqn
 where $k(n)$ is the tuple of the exponents of primes appearing in the prime factorization of $n$ and $\mathscr B_*\nu$ is the push-forward of $\nu$ by the {\it Bohr lift} 
 \beq
 \label{B-lift}
 \mathscr B(s)=(p^{-s}_1, p^{-s}_2, \ldots ), \quad s \in (0, \infty) 
 \eeq
 with $\{p_1, p_2, \ldots \}$ denoting for the monotone enumeration of the prime numbers (see \cite[Section~5]{PP} for more details). The $[0, 1]^{\infty}$-moment problem above is a particular instance of the Hausdorff moment problem in an infinite number of variables (see \cite[Theorems~5.1 and 5.5]{AJK} and \cite[Theorem~3.8]{GKM} for variants of Riesz-Haviland theorem for this moment problem).
 These instances together with the previous discussion motivates the following definition (see \cite[Chapter~2]{BCR} for the notion of the restriction of a Borel measure to a Borel set):

\begin{definition}	
Fix a positive integer $j.$	A sequence of non-negative real numbers $\mathbf w = \{w_{n}\}_{n \geqslant j}$ is called a {\it Hausdorff log-moment sequence} if there exists a positive Borel measure $\mu$ concentrated on $[0,1]$ such that the restriction $\mu|_{(0, 1]}$ of $\mu$ to $(0, 1]$ is a Radon measure and
	\begin{equation*}
		w_{n} = \int_{[0, 1]} t^{\log(n)} \mu(dt), \quad     n \geqslant j.
	\end{equation*}
	We refer to $\mu$ as a {\it representing measure} of $\mathbf w.$ 
	\end{definition}
\begin{remark} \label{rmk-cone} For a Hausdorff log-moment sequence $\mathbf w = \{w_{n}\}_{n \geqslant j}$ with a representing measure $\mu,$ we note the following:
\begin{enumerate}
\item if $j=1,$ then $\mu$ is a finite measure with total density equal to $w_1,$ 
\item if $\mu$ is a representing measure of $\mathbf w$ such that $\mu(\{0\})$ is finite, then since $\mu|_{(0, 1]}$ is locally finite, $\mu$ is a $\sigma$-finite measure,
\item for every integer $n \geqslant j,$ $w_n \leqslant w_j$  $($since $n \mapsto t^{\log(n)}$ is decreasing for every $t \in [0, 1]).$ Thus, the sequence $\mathbf w$ is decreasing and bounded.
\end{enumerate}
\end{remark}

In case $j \Ge 2,$ representing measures of Hausdorff log-moment sequences are not necessarily finite (cf. \cite[Section~1]{Mc-1}). 
\begin{example} \label{Exm-DMS-1}
Let us see some families of Hausdorff log-moment sequences. 
	\begin{itemize}
\item[(a)] For a real number $\alpha < 0,$ the sequence $\{(\log(n))^{\alpha}\}_{n \geqslant 2}$ is a Hausdorff log-moment sequence with representing measure $\mu_\alpha$ equal to 
\beqn
\mu_\alpha(dt) = \frac{1}{\Gamma(-\alpha)} \frac{(-\log(t))^{-1-\alpha}}{t}\,dt,
\eeqn
where $\Gamma$ denotes the Gamma function defined on the right half plane $\mathbb H_0.$
To see this, note that $\mu_{\alpha}|_{_{(0, 1]}}$ is a Radon measure.
Moreover, by Lemma~\ref{CoV}$($ii$)$ $($with $\varphi$ given by \eqref{phi-psi}$),$ 
for any positive integer $n \geqslant 2,$
\beqn
 \int_{[0, 1]} t^{\log(n)} \mu_\alpha(dt) &=& 
 \int_{[0, \infty)} e^{-s \log(n)} \varphi_* \mu_{\alpha}(ds) \\ &=&
  \frac{1}{\Gamma(-\alpha)} \int_{[0, \infty)}  e^{-s\log(n)} s^{-1-\alpha} \,ds  \\
 &=&	\frac{(\log(n))^{\alpha}}{\Gamma(-\alpha)} \int_{[0, \infty)}  e^{-r} r^{-1-\alpha} \,dr \\ 
 &=& (\log(n))^{\alpha}.
\eeqn
	\item[(b)] For a real number $\alpha > 0,$ the sequence $\big\{\frac{1}{\log(n) + \alpha}\big\}_{n \geqslant 1}$ is a Hausdorff log-moment sequence with the representing measure given by $t^{\alpha-1}dt.$
Indeed, for any positive integer $n \geqslant 1,$
\beqn
\int_{[0, 1]} t^{\log(n)}t^{\alpha-1}dt = 
\left.\frac{t^{\log(n) + \alpha}}{\log(n)+\alpha} \,\right\vert_{0}^1  
= \frac{1}{\log(n) + \alpha}.
\eeqn
\item[(c)] For a real number $\alpha \in [0,1]$, the sequence $\{\alpha^{\log(n)}\}_{n \geqslant 1}$ is a Hausdorff log-moment sequence with the representing measure equal to the point mass measure $\delta_{\alpha}$ at $\alpha$.
 Indeed,
\beqn
\alpha^{\log(n)} = \int_{[0, 1]} t^{\log(n)}\delta_{\alpha}(dt), \quad n  \geqslant 1.		
\eeqn
In particular, for a non-negative real number $p$, the sequence $\{\frac{1}{n^{p}}\}_{n \geqslant 1}$ is a Hausdorff log-moment sequence with the representing measure equal to the atomic measure $\delta_{e^{-p}}$ with point mass at $e^{-p}$ $($the case in which $\alpha = e^{-p}).$
Moreover, $\{1,0,0, \ldots\}$ is a Hausdorff log-moment sequence with the representing measure equal to the atomic measure $\delta_{0}$ with point mass at $0$ $($the case in which $\alpha =0).$  \hfill $\diamondsuit$
\end{itemize}
\end{example}	
	
Our answer to Question~\ref{Q1} relies on the following characterization of the Hausdorff log-moment sequences given in terms of the associated linear functional on a certain subspace of the space of continuous functions on $[0, 1].$
\begin{proposition} \label{Q1-DS}
Let $\mathbf w:= \{w_{n}\}_{n\geqslant 1}$ be a sequence of non-negative real numbers. 
Let $\mathcal E[t]$ denote the real linear span of functions $f_{n}(t) = t^{\log(n)},$ $t \in [0, 1],$ $n \Ge 1.$ 
Consider the real linear functional $R_{\mathbf w}: \mathcal E[t] \longrightarrow \mathbb{R}$ defined by setting ${R}_{\mathbf w}(f_{n}) = w_{n},$ $n \geqslant 1$ and extended linearly to $\mathcal E[t].$
	For a closed subset $K$ of $[0, 1],$ 
the following statements are equivalent:
	\begin{itemize}
		\item[(i)] $R_{\mathbf w}$ is $\mathcal E^K_{+}[t]$-positive, 
		\item[(ii)] $R_{\mathbf w}$ is a $K$-moment functional, 
		\item[(iii)] there exists a finite positive Radon measure $\mu$ concentrated on $K$ such that
		\beq \label{D-seq}
		w_{n} = \int_{K} t^{\log(n)} \mu(dt), \quad     n \geqslant 1,
	\eeq
		\item[(iv)] $\mathbf w$ is a Hausdorff log-moment sequence with a representing measure concentrated on $K.$
	\end{itemize}
	\end{proposition}
	\begin{proof}
		Note that the constant function $1$ belongs to $\mathcal E[t].$ Hence, if the linear functional $R_{\mathbf w}$ is $\mathcal E^K_{+}[t]$-positive, then by \cite[Proposition~1.9]{Sc}, there exists a positive Radon measure $\mu$ concentrated on $K$ such that
\beq \label{R-w}
R_{\mathbf w}(q) = \int_{K} q(t) \mu(dt), \quad q \in \mathcal E[t].
\eeq
Thus we obtain the implication (i)$\Rightarrow$(ii).
 Letting $q=f_n$ in \eqref{R-w}, we get
		\eqref{D-seq} and letting $n=1$ in \eqref{D-seq}, we get $\mu(K)=w_1.$	 This yields the implication (ii)$\Rightarrow$(iii). The implication (iii)$\Rightarrow$(iv) is trivial. 
		To see the implication (iv)$\Rightarrow$(i), note that if $\mathbf w$ is a Hausdorff log-moment sequence, then by the linearity of the integral,
		$$R_{\mathbf w}(f) = \int_{K} f(x) \mu(dx), \quad f \in \mathcal E[t].$$
		It now follows that $R_{\mathbf w}$ is $\mathcal E^K_{+}[t]$-positive.
	\end{proof}
	
	Proposition~\ref{Q1-DS} can be used to prove an analogue of Riesz-Haviland theorem for the space of Dirichlet polynomials. We present two proofs of this analogue. The first one relevant to the study of  Hausdorff log-moment sequences is presented in Section~3. The second one based on a known compactification technique is given in Appendix. 

	\section{An analogue of the Riesz-Haviland theorem for $\mathcal D[s]$}
	
For a subset $\sigma$ of the set $\mathbb Z_+$ of positive integers, let $\chi_{\sigma}$ denote the indicator function of $\sigma.$ We sometimes use  $\chi_\sigma$ to denote the sequence $\{\chi_\sigma(n)\}_{n \Ge 1}.$ 	

The following result characterizes $\mathcal D^{K}_{+}[s]$-positive linear functionals $L_{\mathbf w}$ answering Question~\ref{Q1} (cf. \cite[Theorem~5.1]{AJK}).
\begin{theorem} \label{Q1-DS-Coro}
 	Let $\mathbf w = \{w_{n}\}_{n\geqslant 1}$ be a sequence of non-negative real numbers. 
 	Consider the real linear functional $L_{\mathbf w}: \mathcal D[s] \longrightarrow \mathbb{R}$ defined by $L_{\mathbf w}(n^{-s}) = w_{n},$ $n \geqslant 1,$ and extended linearly to $\mathcal D[s].$
For a closed subset $K$ of $[0,\infty),$ the following statements are equivalent:
 	\begin{itemize}
 		\item[(i)] $L_{\mathbf w}$ is $\mathcal D^{K}_{+}[s]$-positive,
 		\item[(ii)] there exists a finite positive Radon measure $\nu$ concentrated on $K$ with $\nu(K) \Le w_1$ such that for every $p \in \mathcal D[s],$
 	\beq 	\label{common-exp-0}
 			L_{\mathbf w}(p) 
 			= \begin{cases} 
 			\displaystyle	 (w_1 - \nu(K))  \lim_{s \rar \infty}p(s) + \int_{K} p(s) \nu(ds) & \mbox{if~} \text{$K$ is unbounded,} \\
 	\displaystyle			\int_{K} p(s) \nu(ds) & \mbox{if~} \text{$K$ is bounded,}
 			\end{cases}
 	\eeq
 		\item[(iii)] there exists a finite positive Radon measure $\nu$ concentrated on $K$ with $\nu(K) \Le w_1$ such that for every positive integer $n \geqslant 1,$
 		\beq
 		\label{common-exp}
 		w_n = \begin{cases}
 	\displaystyle		(w_1-\nu(K)) \chi_{_{\{1\}}}\!(n) + \int_{K} n^{-s}\nu(ds) & \mbox{if~} \text{$K$ is unbounded,} \\
 	\displaystyle	\int_{K} n^{-s}\nu(ds) & \mbox{if~} \text{$K$ is bounded.}
 		\end{cases}
 		\eeq
 	\end{itemize}
 \end{theorem}
 
We do not see any obvious way to derive Theorem~\ref{Q1-DS-Coro} from \cite[Theorem~5.1]{AJK}. Also, Theorem~\ref{Q1-DS-Coro} falls short to deduce (via the change of variable using Bohr lift given by \eqref{B-lift}) the Riesz-Haviland theorem for the moment problem in an infinite number of variables. 

 In the proof of Theorem~\ref{Q1-DS-Coro}, we need a couple of lemmas. We recall first the notion of an image measure.
	
Let $(X, \Sigma, \mu)$ be a measure space and $Y$ be a Hausdorff space. Let $\mathcal B(Y)$ denote the Borel $\sigma$-algebra of $Y.$ If $\phi : X \rar Y$ is a 
$\Sigma$-measurable function, then $\phi_*\mu$ denotes the {\it push-forward of $\mu$ by $\phi$} or the {\it image of $\mu$ under $\phi$} defined as
$\phi_*\mu(\sigma) = \mu(\phi^{-1}(\sigma)),$ $\sigma \in \mathcal B(Y).$
The reader is referred to \cite[Chapter~2.1]{BCR} for elementary facts pertaining to the image measures. We find it convenient to state the following consequence of \cite[Exercise 1.4.38]{T}.
\begin{lemma}[Change of variables] 
\label{CoV}
Let $\mu$ be a positive Borel measure on $[0, 1]$ and let $\nu$ be a positive Borel measure on $[0, \infty).$ 
Define functions  $\varphi :  (0,1] \longrightarrow [0, \infty)$ and $\psi :  [0, \infty) \longrightarrow (0, 1]$ by 
\beq \label{phi-psi}
\varphi(t) = - \log(t), ~ t \in (0, 1], \quad 
\psi(s)  = e^{-s}, ~ s \in [0, \infty).
\eeq
Then, for Borel measurable subsets $\widetilde{K}$ and $K$ of $[0, 1]$ and $[0, \infty)$ respectively, the following statements are valid.
\begin{enumerate}
\item[(i)] if $0 \notin \widetilde{K},$ then for every $\lambda \in [0, \infty),$
\beqn
\int_{\widetilde{K}} t^{\lambda} \mu(dt) = \int_{\varphi(\widetilde{K})} e^{-s \lambda} \varphi_* \mu(ds), 
\eeqn  
\item[(ii)] for every $\lambda \in [0, \infty),$
\beqn
 \displaystyle \int_{K} e^{-s \lambda} \nu(ds) = \int_{\psi(K)} t^{\lambda} \psi_*\nu(dt).
\eeqn
\end{enumerate}
Both sides in the above identities could be possibly infinite.
\end{lemma}

The following lemma provides precise relationship between $\mathcal D[s]$ (resp. $\mathcal D^K_+[s]$) and $\mathcal E[t]$ (resp. $\mathcal E^{\widetilde{K}}_+[t]$).
\begin{lemma} \label{relationship}
 	Let $\mathcal E[t]$ and $\mathcal E^K_+[t]$ be as defined in the statement of Proposition~\ref{Q1-DS}.
 	If $\varphi :  (0,1] \longrightarrow [0, \infty)$ and $\psi :  [0, \infty) \longrightarrow (0, 1]$ are given by \eqref{phi-psi}, then the following statements are valid:
 	\begin{enumerate}
 		\item[(i)] if $q \in \mathcal E[t],$ then $q \circ \psi \in \mathcal D[s],$
 		\item[(ii)] if $p \in \mathcal D[s],$ then there exists $q_p \in \mathcal E[t]$ such that $p = q_p \circ \psi,$
 		where $q_p$ is given by
 		\beq \label{q}
 		q_p(t) = \begin{cases} \displaystyle \lim_{s \rar \infty}p(s) & \mbox{if~} t=0, \\
 			p \circ \varphi(t), & \mbox{if}~t \in (0, 1], 
 		\end{cases}
 		\eeq
 		\item[(iii)] if $\widetilde{K}$ is a closed subset of $[0, 1]$ and $q \in \mathcal E[t],$ then $q \in \mathcal E^{\widetilde{K}}_+[t]$ if and only if $q \circ \psi \in \mathcal D^{{K}}_+[s],$ where ${K} = \psi^{-1}(\widetilde{K} \backslash \{0\}).$
 	\end{enumerate}
 \end{lemma}
 \begin{proof}
 	For every integer $n \Ge 1,$ if $q(t)=t^{\log(n)},$ $t \in [0, 1],$ then $q \circ \psi(s) = n^{-s},$ $s \in [0, \infty).$  The statement (i) is now clear. To see (ii), let $p \in \mathcal D[s]$ and note that $\lim_{s \rar \infty}p(s)$ exists as a real number. Thus $q_p$ given by \eqref{q} is well-defined. 
 	Since $\varphi \circ \psi$ is the identity function on $[0, \infty),$ 
  $p = q_p \circ \psi.$ The equivalence in (iii) follows from (i) and the continuity of $q.$
 \end{proof}

 \begin{proof}[First proof of Theorem~\ref{Q1-DS-Coro}] 
 In view of \cite[Proposition~1.9]{Sc}, we may assume that $K$ is unbounded. 
 Let $\varphi :  (0,1] \longrightarrow [0, \infty)$ and $\psi :  [0, \infty) \longrightarrow (0, 1]$ be given by \eqref{phi-psi}.
 	To see the implication (i)$\Rightarrow$(ii), suppose that $L_{\mathbf w}$ is a $\mathcal D^{K}_{+}[s]$-positive linear functional satisfying 
 	\beq 
 	\label{value-L}
 	L_{\mathbf w}(n^{-s}) = w_{n}, \quad n \geqslant 1.
 	\eeq
 	Let $\mathcal E[t]$ and $\mathcal E^{K}_+[t]$ be as defined in the statement of Proposition~\ref{Q1-DS}.
 	Define $R_{\mathbf w}: \mathcal E[t] \longrightarrow \mathbb{R}$ by setting $R_{\mathbf w}(t^{\log(n)}) = w_n,$ and extend it linearly to $\mathcal E[t].$ Let $\widetilde{K} = \overline{\phi^{-1}(K)}$ denote a closed subset of $[0, 1]$ and note that $\widetilde{K} = \phi^{-1}(K) \cup \{0\}$ (see \eqref{phi-psi}).
 	We now check that $R_{\mathbf w}$ is $\mathcal E^{\widetilde{K}}_+[t]$-positive.
 	By Lemma~\ref{relationship}(iii), if $q \in \mathcal E^{\widetilde{K}}_+[t],$ then $q \circ \psi \in \mathcal D^{K}_+[s]$ and 
 	\beqn
 	R_{\mathbf w}(q) \overset{\eqref{value-L}}= L_{\mathbf w}(q \circ \psi) \geqslant 0.
 	\eeqn
 	It now follows from Proposition~\ref{Q1-DS} that
 	there exists a finite positive Radon measure $\mu$ concentrated on $\widetilde{K}$ such that
 	\begin{equation*}
 		L_{\mathbf w}(q \circ \psi)= \int_{\widetilde{K}} q(t) \mu(dt), \quad   q \in \mathcal E[t].
 	\end{equation*}
 	By Lemma~\ref{relationship}(ii), for every $p \in \mathcal D[s],$ there is $q_p \in \mathcal E[t]$ such that 
 	\beqn
 	L_{\mathbf w}(p)
 	&=&  L_{\mathbf w}(q_p \circ \psi) \\
 	&\overset{\eqref{q}}=&  \displaystyle  \mu(\{0\})  \lim_{s \rar \infty}p(s)  + \int_{\phi^{-1}(K)} q_p(t) \mu(dt).
 	\eeqn
 By Lemma~\ref{CoV}(i) and \eqref{q}, for every $p \in \mathcal D[s],$
 	\beq \label{proof-L-mom}
 	L_{\mathbf w}(p) =  \displaystyle \mu(\{0\}) \lim_{s \rar \infty}p(s)   + \int_{K} p(s) \varphi_*\mu(ds).
 	\eeq
 	Since $\mu$ is a finite Radon measure, so is $\varphi_*\mu$ (see \cite[Proposition~2.1.15]{BCR}). 
 	This completes the verification of (i)$\Rightarrow$(ii) with $\nu =\varphi_*\mu$
 	provided we check that $\nu$ is a finite measure satisfying $\nu(K) \Le w_1.$  
To see this, let $p(s)=n^{-s},$ $s \in [0, \infty)$ in \eqref{proof-L-mom} to obtain
 	\beq 
 \label{2-1-imp}	w_n = L_{\mathbf w}(p)=  \mu(\{0\})\,\chi_{_{\{1\}}}\!(n)  + \int_{K} n^{-s} \varphi_*\mu(ds).
 	\eeq
 	This yields  
 	\beq \label{nu-mu-rel-1}
 	\mu(\{0\}) = w_1-\varphi_*\mu(K)
 	\eeq 
 	completing the verification of (i)$\Rightarrow$(ii). 
 	The implication (ii)$\Rightarrow$(iii) is now immediate from \eqref{2-1-imp}.
 	
 	To see the implication (iii)$\Rightarrow$(i), 
 	note that for any $p \in \mathcal D[s],$
 	\beqn
 	L_{\mathbf w}(p) 
   	= 
 \displaystyle		(w_1 - \nu(K)) \lim_{s \rar \infty}p(s)    + \int_{K} p(s) \nu(ds).
 	\eeqn
 	Since $\nu(K) \Le w_1$  and $\lim_{s \rar \infty} p(s) \Ge 0,$ $p \in \mathcal D_+[s],$ $L_{\mathbf w}$ is $\mathcal D_{+}[s]$-positive.
 \end{proof}

 Here is an application of Theorem~\ref{Q1-DS-Coro} to Hausdorff log-moment sequences. 
 \begin{corollary} \label{Coro-HLM}
 	Let $\mathbf w = \{w_{n}\}_{n\geqslant 1}$ be a sequence of non-negative real numbers. 
 	Consider the real linear functional $L_{\mathbf w}: \mathcal D[s] \longrightarrow \mathbb{R}$ defined by $L_{\mathbf w}(n^{-s}) = w_{n},$ $n \geqslant 1,$ and extended linearly to $\mathcal D[s].$
Then $L_{\mathbf w}$ is $\mathcal D^{[0, \infty)}_{+}[s]$-positive if and only if $\mathbf w$ is a Hausdorff log-moment sequence.
\end{corollary}
\begin{proof}
Applying Theorem~\ref{Q1-DS-Coro} to the closed set $K=[0, \infty)$ yields that $L_{\mathbf w}$ is $\mathcal D^{[0, \infty)}_{+}[s]$-positive if and only if there exists a finite positive Radon measure $\nu$ concentrated on $[0, \infty)$ with $\nu(K) \Le w_1$ such that 
 		\beqn
 		w_n = 
 (w_1-\nu(K)) \chi_{_{\{1\}}}\!(n) + \int_{K} n^{-s}\nu(ds), \quad n \Ge 1. 
 		\eeqn
The desired equivalence is now immediate from Example~\ref{Exm-DMS-1}(c) and Lemma \ref{CoV}(ii). In this case, one can choose the representing measure of $\mathbf w$ to be $(w_1-\nu(K))\delta_0 + \psi_*\nu.$ 
\end{proof}

As noted in \cite[Exercise~1.3.3]{Sc}, the evaluation functional at the point in the Stone–$\check{\mbox{C}}$ech compactification of $\mathbb R$ and not belonging to $\mathbb R$ is a $C(\mathbb R, \mathbb R)_+$-positive functional which is not a $\mathbb R$-moment functional. This fact together with the second proof of Theorem~\ref{Q1-DS-Coro} suggests that the perturbation to the integral appearing in \eqref{common-exp-0} could be an obstruction in making $L_{\mathbf w}$ a $K$-moment functional. This is confirmed by the next result, which also answers Question~\ref{Q2}. 

\begin{corollary} \label{rmk-Q1}
	Consider the real linear functional $L_{\mathbf w}: \mathcal D[s] \longrightarrow \mathbb{R}$ defined by $L_{\mathbf w}(n^{-s}) = w_{n},$ $n \geqslant 1,$ and extended linearly to $\mathcal D[s].$
For a closed unbounded subset $K$ of $[0, \infty),$	assume that $L_{\mathbf w}$ is $\mathcal D^K_{+}[s]$-positive, so 
that $L_{\mathbf w}$ is given by \eqref{common-exp-0} for some finite positive Radon measure $\nu$ concentrated on $K.$ 
Then the following statements are equivalent:
\begin{enumerate}
\item[(i)] $L_{\mathbf w}$ is a $K$-moment functional, 
	\item[(ii)]  $w_1 = \nu(K),$
	\item[(iii)] $\mathbf w$ is a Hausdorff log-moment sequence with a representing measure $\mu$ concentrated on $\widetilde{K}$ such that $0 \in \widetilde{K}$ and $\mu(\{0\}) = 0.$
	\end{enumerate}
\end{corollary}
\begin{proof} 
(i)$\Leftrightarrow$(ii):   Let $L_{\mathbf w}$ be a $K$-moment functional with a representing measure $\gamma.$  Combining (i) with \eqref{common-exp}, we obtain
	\begin{equation*}
		(w_1-\nu(K)) \, \chi_{_{\{1\}}}\!(n) = \int_{K} n^{-s} \gamma(ds) -  \int_{K} n^{-s} \nu(ds), \quad   n \geqslant 1.
	\end{equation*}
By Lemma~\ref{CoV}(ii) $($with $\psi$ given by \eqref{phi-psi}$),$ 
\begin{equation*}
		(w_1-\nu(K)) \, \chi_{_{\{1\}}}\!(n) = \int_{\psi(K)} t^{\log(n)} \psi_*\gamma(dt) -  \int_{\psi(K)} t^{\log(n)} \psi_*\nu(dt), \quad   n \geqslant 1.
	\end{equation*}
	Let $\delta_0$ denote the atomic measure with point mass at $\{0\}$ and $\eta$ denote the trivial extension of $\psi_*\gamma - \psi_*\nu$ to $[0, 1].$  Then
	 \beq \label{L-w-proof}
		(w_1-\nu(K))  \int_{[0, 1]} t^{\log(n)} \delta_0(dt)= \int_{[0, 1]} t^{\log(n)} \eta(dt),  \quad  n \geqslant 1.
	\eeq
By the Stone-Weierstrass theorem (see \cite[Theorem 2.5.5]{Si-1}), the subspace $\mathcal S = \mbox{span}_{\mathbb R}\{t^{\log(n)}\}_{n \Ge 1}$ is dense in the space $C([0,1], \mathbb R)$ of real-valued continuous functions of $[0, 1].$ Hence, the identity \eqref{L-w-proof} holds for any $f \in C([0,1], \mathbb R).$ If $w_1 \neq \nu(K),$ then
by the Riesz-Markov theorem (see \cite[Theorem 4.8.8]{Si-1}), we obtain
$$(w_1-\nu(K))\delta_0 = \eta.$$
However, in that case, $\eta$ would be supported only at $\{0\},$ which is contrary to the definition of $\eta.$ This yields $w_1 = \nu(K)$ completing the verification of (i)$\Rightarrow$(ii).
 Clearly, if $w_1 = \nu(K),$ then by \eqref{common-exp-0}, $L_{\mathbf w}$ is a $K$-moment functional with representing measure $\gamma = \nu.$ 

 (ii)$\Leftrightarrow$(iii): If $w_1 = \nu(K),$ then by \eqref{common-exp}, 
 	\beqn
 w_n =  \int_{K} n^{-s}\nu(ds), \quad n \geqslant 1.
 \eeqn
 Since $0 \notin \psi(K)$ (see \eqref{phi-psi}), one can extend ${\psi_* \nu}$ trivially to a measure $\mu$ concentrated on $\psi(K) \cup \{0\}$ such that 
 $\mu(\{0\})=0.$
 One may now conclude from Lemma~\ref{CoV}(ii) that $\mathbf w$ is a Hausdorff log-moment sequence with representing measure $\mu.$ 
 On the other hand, if $\mathbf w$ is a Hausdorff log-moment sequence with representing measure $\mu$ concentrated on $\widetilde{K}$ such that $0 \in \widetilde{K}$ and $\mu(\{0\}) =0,$ then by \eqref{nu-mu-rel-1} (see the first proof of Theorem~\ref{Q1-DS-Coro}), we obtain  $$\mu(\{0\}) = w_1-\varphi_*\mu(K),$$  and hence $w_1=\varphi_*\mu(K).$
\end{proof}

It is well-known that the Riesz-Haviland theorem can be employed to obtain the solution of the multi-dimensional Hausdorff moment problem (see \cite[Section 3.2]{Sc}; for different variants of Riesz-Haviland theorem, see \cite[Theorem~1.1]{CF}, \cite[Theorems~A and B]{CSS} and \cite[Theorem~2.6]{Am}). Thus, in view of Proposition~\ref{Q1-DS} and Theorem~\ref{Q1-DS-Coro}, it is natural to look for a solution of the Hausdorff log-moment problem. Although we could not obtain an intrinsic characterization of Hausdorff log-moment sequences (unlike the case of Hausdorff moment sequences; see \cite[Theorem 4.17.4]{Si-1}), it is possible to obtain a handy characterization of these sequences that exploits the theory of completely monotone functions  (we would like to draw attention of the reader to \cite{Liu}, which explores a connection between Dirichlet series and completely monotone functions in the context of an approximation problem). 

	\begin{theorem} \label{thm-char-DMS}
For a positive integer $j,$ let $\{w_{n}\}_{n \geqslant j}$ be a sequence of non-negative real numbers. Then the following statements are valid:
\begin{enumerate}
\item[(i)] $\{w_{n}\}_{n \geqslant 1}$ is a Hausdorff log-moment sequence if and only if 
there exists a unique completely monotone function $f : [0,\infty) \rar [0, \infty)$ such that $f(0) \leqslant w_1$ and
\beq \label{DMS-formula-1}
w_n = (w_1-f(0)) \, \chi_{_{\{1\}}}(n) + f(\log(n)),    \quad n \geqslant 1,
\eeq
where $\chi_{_{\{1\}}} : \mathbb Z_+ \rar \mathbb R$ denotes the indicator function of $\{1\}.$ 	
If $\{w_{n}\}_{n \geqslant 1}$ is a Hausdorff log-moment sequence with the representing measure $\mu$, then we may choose $f(0)$ to be $w_1-\mu(\{0\}).$ 
\item[(ii)] for $j \geqslant 2,$ $\{w_{n}\}_{n \geqslant j}$ is a Hausdorff log-moment sequence if and only if there exists a unique completely monotone function $f : [\log(j),\infty) \rar [0, \infty)$ such that 
\beq \label{DMS-formula-2}
w_n = f(\log(n)),    \quad n \geqslant j.
\eeq
\end{enumerate}
If this happens, then the representing measure of $f$ is equal to $\varphi_*\mu$ $($see \eqref{phi-psi}$).$
\end{theorem}

It is natural to compare the formulas \eqref{common-exp} and \eqref{DMS-formula-1}, and it is tempting to explore the possibility to deduce the existence part of Theorem~\ref{thm-char-DMS}(i) from Theorem~\ref{Q1-DS-Coro}. Unfortunately, the method of proof of Theorem~\ref{Q1-DS-Coro} does not extend beyond the case $j=1.$ 

The proof of Theorem~\ref{thm-char-DMS} is fairly long and it occupies most of Sections~4 and 5. It is worth noting that the proof of the uniqueness part of this theorem relies on Carlson-Fuchs's uniqueness theorem for holomorphic functions on the right half plane of exponential type (see \cite{Bo, Fu}). 
In Section~5, we also discuss applications of Theorem~\ref{thm-char-DMS} to the Helson matrices and the completely monotone sequences (see Corollaries~\ref{coro-Helson} and \ref{coro-CM}).

\section{Basic properties}	
In this section, we closely examine Hausdorff log-moment sequences. In particular, we present  several algebraic and structural properties of the Hausdorff log-moment sequences. Some of these properties will be used in the proof of Theorem~\ref{thm-char-DMS}.

%

The following proposition provides some algebraic properties of the Hausdorff log-moment sequences:
\begin{proposition} \label{prop-DMS-properties}
For a positive integer $j,$ let $\mathbf w = \{w_{n}\}_{n \geqslant j}$ be a Hausdorff log-moment sequence with a representing measure $\mu$ and $\mathbf v = \{v_{n}\}_{n \geqslant j}$ be a Hausdorff log-moment sequence with representing measure  $\nu.$
The following statements are valid:
\begin{itemize}
		\item[(i)] 
$\mathbf w + \mathbf v$ is a Hausdorff log-moment sequence with representing measure  $\mu+\nu,$
	\item[(ii)] if $c \geqslant 0$ is a real number, then $c \mathbf w$ is a Hausdorff log-moment sequence with representing measure $c\mu,$ 
		\item[(iii)]	 if $\mu(\{0\}) =0$ and $\nu(\{0\}) =0,$ then the point-wise product $\mathbf w \mathbf v $ of $\mathbf w$ and $\mathbf v$ is a Hausdorff log-moment sequence with representing measure equal to $f_* (\mu|_{_{(0, 1]}} \otimes \nu|_{_{(0, 1]}}),$ where $f : [0,1] \times [0,1] \longrightarrow [0,1]$ is given by $f(x,y) = xy$ and $\mu|_{_{(0, 1]}} \otimes \nu|_{_{(0, 1]}}$ is a uniquely determined Radon measure concentrated on $(0, 1] \times (0, 1]$ such that the restriction $\mu|_{_{(0, 1]}} \otimes \nu|_{_{(0, 1]}}$ to the product $\sigma$-algebra coincides with the product measure $\mu|_{_{(0, 1]}} \times \nu|_{_{(0, 1]}}.$ 
	\end{itemize}
	In particular, the set of Hausdorff log-moment sequences forms a convex cone.
\end{proposition}
\begin{proof}
We leave the verification of (i) and (ii) to the reader. 

To see (iii), assume that
$\mu(\{0\})=0$ and $\nu(\{0\})=0.$ By Remark~\ref{rmk-cone}(2), $\mu$ and $\nu$ are $\sigma$-finite measures. 
By \cite[Corollary~2.1.11]{BCR}, there exists a uniquely determined Radon measure $\mu|_{_{(0, 1]}} \otimes \nu|_{_{(0, 1]}}$ such that the restriction of $\mu|_{_{(0, 1]}} \otimes \nu|_{_{(0, 1]}}$ to the product $\sigma$-algebra coincides with the product measure $\mu|_{_{(0, 1]}} \times \nu|_{_{(0, 1]}}.$ 
Note that for any integer $n \Ge j$ and $g : [0,1] \longrightarrow [0,1]$ given by $g(x) = x^{\log(n)}$, 
\allowdisplaybreaks
	\beqn
			w_{n}v_{n} & =& \int_{[0,1]} \int_{[0,1]} (xy)^{\log(n)} \mu(dx) \nu(dy)\\
			& = & \int_{(0,1]} \int_{(0,1]} g\circ f(x,y)  \mu|_{_{(0, 1]}}(dx) \nu|_{_{(0, 1]}}(dy)\\
			& \overset{(*)}=& \int_{(0,1] \times (0,1]} g\circ f(x, y)    \mu|_{_{(0, 1]}} \otimes \nu|_{_{(0, 1]}}(d(x, y))\\
			& =& \int_{(0,1]} g(x) f_* (\mu|_{_{(0, 1]}} \otimes \nu|_{_{(0, 1]}})(dx) \quad (\mbox{by change of variables})  \\ 
			& = & \int_{[0,1]} x^{\log(n)} f_* (\mu|_{_{(0, 1]}} \otimes \nu|_{_{(0, 1]}})(dx),
\eeqn
		where ($\ast$) follows from an analogue of Fubini-Tonelli theorem for Radon measures (see \cite[Theorem~2.1.12]{BCR}).
Since $f_* (\mu|_{_{(0, 1]}} \otimes \nu|_{_{(0, 1]}})\big|_{(0, 1]}$ is a Radon measure (see \cite[Proposition~2.1.15]{BCR}), $\mathbf w \mathbf v$ is a Hausdorff log-moment sequence with representing measure $f_* (\mu|_{_{(0, 1]}} \times \nu|_{_{(0, 1]}}).$
\end{proof}

The following proposition provides some structural properties of Hausdorff log-moment sequences (cf. \cite[Theorem 4.17.4]{Si-1}). The reader is referred to \cite{BCR, SSV, W} for the definitions of positive semi-definite matrix, completely monotone sequence, minimal completely monotone sequence, completely monotone function, Bernstein function and related notions.
\begin{proposition} \label{prop-DMS-properties-1}
For a positive integer $j,$ let $\{w_{n}\}_{n \geqslant j}$ be a Hausdorff log-moment sequence with a representing measure $\mu$ concentrated on $[0, 1].$ The following statements are valid: 
\begin{itemize}
		\item[(i)]  for every finite subset $F$ of integers $n \geqslant j$, the matrix  $(w_{pq})_{p, q \in F}$ is positive semi-definite,
		\item[(ii)] if $k \geqslant j$ is an integer such that $k \geqslant j,$ then $\{w_{k^{m+1}}\}_{m \geqslant 0}$ is a completely monotone sequence,
		\item[(iii)] for every integer $k\geqslant j,$ $\{w_{kn}\}_{n \geqslant 1}$ is a Hausdorff log-moment sequence with representing measure given by $t^{\log(k)} \mu(dt).$
	\end{itemize}
\end{proposition}
\begin{proof}


(i) For any complex numbers  $c_1, \ldots, c_n$ and integers $k_1, \ldots, k_n$ bigger than or equal to $j,$
\beqn 
\sum_{p, q = 1}^{n} c_{p}\bar{c_{q}} \, w_{_{k_pk_q}} &=& \int_{[0, 1]} \sum_{p, q = 1}^{n} c_{p}\bar{c_{q}} \, t^{\log(k_pk_q)} \mu(dt) \\ &=& \int_{[0, 1]} \Big|\sum_{p = 1}^{n} c_{p} \, t^{\log(k_p)}\Big|^{2} \mu(dt),
\eeqn
which is clearly non-negative. 
		
(ii) 
For integers $n \geqslant 0$ and $m \Ge 1,$
\beqn 
\sum_{i = 0}^{n} (-1)^{i} {n \choose i} w_{{k}^{m+i}}  &=& \int_{[0, 1]} \sum_{i = 0}^{n} (-1)^{i} {n \choose i} t^{(m+i) \log(k)} \mu(dt) \\
&=& \int_{[0, 1]} t^{m \log(k)} (1-t^{\log(k)})^{n} \mu(dt),    
\eeqn
which is non-negative, as required.

(iii)  For any integer $k \geqslant j,$ 
\beqn w_{kn} = \int_{[0,1]} t^{\log(kn)} \mu(dt) = \int_{[0,1]} t^{\log(n)} t^{\log(k)}\mu(dt), \quad n \geqslant 1.
\eeqn
Thus $\{w_{kn}\}_{n \geqslant 1}$ is a Hausdorff log-moment sequence with representing measure given by $t^{\log(k)}\mu(dt).$
\end{proof}

Any Hausdorff log-moment sequence $\{w_{n}\}_{n \geqslant 1}$ is {\it determinate}, that is, its representing measure is unique. Indeed, we have following general fact (see Corollary~\ref{determinate} for an improvement).
	\begin{proposition}[Uniqueness] 
\label{unique}	Let $j$ be a positive integer. If $\mu$ and $\nu$ are two finite representing measures for $\{w_{n}\}_{n \geqslant j},$ then $\mu=\nu$ provided they have the same total mass. In particular, if $\{w_{n}\}_{n \geqslant 1}$ is a Hausdorff log-moment sequence, then its representing measure is unique.
\end{proposition}
\begin{proof} Let $\mu$ and $\nu$ be two finite representing measures for $\{w_{n}\}_{n \geqslant j}$ and assume that $\mu([0, 1]) = \nu([0, 1]) < \infty.$ For a positive integer $n,$ consider the function $f_{n}(t) = t^{\log(n)}, t\in[0,1].$
Let $\mathcal S$ denote the real linear span of $\{f_1, f_{j}, f_{j+1}, \ldots, \}.$
	Note that
\beq
\label{unique-rm}
\int_{[0, 1]} f(t)\mu(dt) = \int_{[0, 1]} f(t)\nu(dt), \quad f \in \mathcal S.
\eeq
Since $f_m f_n=f_{mn},$ $m, n \in \mathbb Z_+,$ $\mathcal S$ is a real unital sub-algebra of $C([0, 1], \mathbb R).$ 
Hence, by the Stone-Weierstrass theorem (see \cite[Theorem 2.5.5]{Si-1}), $\mathcal S$ is dense in  $C([0,1], \mathbb R).$ Thus, the identity \eqref{unique-rm} holds for any $f \in C([0,1], \mathbb R).$
The desired conclusion now follows from the Riesz-Markov theorem (see \cite[Theorem 4.8.8]{Si-1}). 
\end{proof}

The notion of so-called minimal Hausdorff log-moment sequences plays a role in the proof of Theorem~\ref{thm-char-DMS} and this notion is similar to the one that appears in the context of completely monotone sequence (refer to \cite[Chapter IV.14]{W}).
\begin{definition}
A Hausdorff log-moment sequence $\mathbf w = \{w_{n}\}_{n \Ge 1}$ is said to be {\it minimal} if for any $\epsilon \textgreater 0,$  
$\mathbf w - \epsilon \chi_{\{1\}}$ is not a Hausdorff log-moment sequence,  where 
$\chi_{_{\{1\}}} : \mathbb Z_+ \rar \mathbb R$ denotes the indicator function of $\{1\}.$
\end{definition}

As in the case of completely monotone sequences, the minimal Hausdorff log-moment sequences can be described easily. 

We now identify minimal Hausdorff log-moment sequences.
	\begin{proposition} 
\label{minimal}	
A Hausdorff log-moment sequence $\mathbf w = \{w_{n}\}_{n \Ge 1}$ is minimal if and only if the representing measure $\mu$ of $\mathbf w$ satisfies $\mu(\{0\})=0.$ 
	\end{proposition}
	\begin{proof}
If $\mathbf w$ is a minimal Hausdorff log-moment sequence with a representing measure $\mu$ such that $\mu(\{0\}) \textgreater 0,$ then	
\beqn w_{1} - \mu(\{0\}) = \mu((0,1]), \quad 
		w_{n} = \int_{(0,1]} t^{log(n)} \mu(dt), \quad n \Ge 2.
\eeqn
		Hence $\mathbf w - \mu(\{0\})\chi_{\{1\}}$ is a Hausdorff log-moment sequence with the representing measure $\widehat{\mu}$ obtained by extending $\mu$ trivially to $[0, 1]$: 
		\beq \label{mu-hat-def} 
\text{$\widehat{\mu}(\sigma)  = \mu(\sigma \setminus \{0\})$ for every Borel subset $\sigma$ of $[0, 1].$}
\eeq
		This contradicts our assumption that $\mathbf w$ is a minimal Hausdorff log-moment sequence.

Conversely, suppose that $\mathbf w$ is a Hausdorff log-moment sequence with a representing measure $\mu.$ 
		If there exists $\epsilon \textgreater 0$ such that $\mathbf w_\epsilon = \mathbf w - \epsilon \chi_{\{1\}}$ is a Hausdorff log-moment sequence with a representing measure $\nu$, then by	Proposition~\ref{prop-DMS-properties} and Example~\ref{Exm-DMS-1}, $\mathbf w_\epsilon 
+ \epsilon \chi_{\{1\}}$ is a Hausdorff log-moment sequence with a representing measure $\nu + \epsilon \delta_0.$ By Proposition~\ref{unique}, $\mu$ must coincide with $\nu + \epsilon \delta_0,$ and hence $\mu(\{0\}) = \nu(\{0\}) + \epsilon > 0.$ 
		\end{proof}	 
		\begin{remark} \label{minimal-rmk}
		 Proposition~\ref{minimal} suggests that for any integer $j \Ge 2,$ any Hausdorff log-moment sequence $\mathbf w = \{w_{n}\}_{n \Ge j}$ can be considered as a minimal Hausdorff log-moment sequence. Indeed, if $\mu$ is a representing measure of $\mathbf w$ and $\widehat{\mu}$ is the trivial extension of $\mu|_{(0, 1]}$ $($see \eqref{mu-hat-def}$),$ 
		  then since for $n \geqslant j > 1,$ $t^{\log(n)}|_{t=0} = 0,$ after replacing $\mu$ by $\widehat{\mu}$ if necessary, we may assume that $\mu(\{0\})=0.$ We used here the convention that $0 \cdot \infty =0.$
\end{remark}

		
We need the following corollary in the proof of Theorem~\ref{thm-char-DMS}.
		\begin{corollary} \label{minimal-coro} Let 
 $\mathbf w = \{w_{n}\}_{n \Ge 1}$ be a sequence of non-negative real numbers. Then the following statements are valid: 
 \begin{enumerate}
\item[(i)] if $\mathbf w$ is a Hausdorff log-moment sequence with the representing measure $\mu,$ then the sequence 
$\mathbf w -  \mu(\{0\}) \, \chi_{_{\{1\}}}$ is a minimal Hausdorff log-moment sequence with representing measure $\widehat{\mu}$ given by \eqref{mu-hat-def},
\item[(ii)] if there exists a real number $c \in [0, w_1]$ such that $\mathbf w -  c \, \chi_{_{\{1\}}}$ is a Hausdorff log-moment sequence with representing measure $\nu,$ then $\mathbf w$ is a Hausdorff log-moment sequence with representing measure $\nu + c \delta_0,$ where $\delta_0$ denote the atomic measure with point mass at $\{0\}.$ 
\end{enumerate}
		\end{corollary}
		\begin{proof}
(i) 	Assume that $\{w_{n}\}_{n \geqslant 1}$ is a Hausdorff log-moment sequence. By Proposition~\ref{Q1-DS}, there exists a finite Radon measure $\mu$ on $[0,1]$ such that
$w_{n} = \int_{[0, 1]}t^{\log(n)} \mu(dt)$ for every integer $n \geqslant 1.$  
Let $\widehat{\mu}$ be as defined in \eqref{mu-hat-def}. Note that $\widehat{\mu}(\{0\})=0$
and $\widehat{\mu}$ is a Radon measure that satisfies 
\beqn
w_n &=& \int_{[0, 1]} t^{\log(n)} \mu(dt) \\
&=& \begin{cases} \mu(\{0\})  + \int_{[0, 1]} t^{\log(n)} \widehat{\mu}(dt) & \mbox{if~}n=1, \\
 \int_{[0, 1]} t^{\log(n)} \widehat{\mu}(dt) & \mbox{if~} n \geqslant 2.\\
 \end{cases}
\eeqn		
Thus the sequence $\mathbf w -  \mu(\{0\}) \, \chi_{_{\{1\}}}$ is a Hausdorff log-moment sequence with representing measure $\widehat{\mu}.$ Since $\widehat{\mu}(\{0\})=0,$ by Proposition~\ref{minimal}, $\mathbf w -  \mu(\{0\}) \, \chi_{_{\{1\}}}$ is a minimal Hausdorff log-moment sequence. 

(ii) Assume that the sequence $\mathbf w -  c \, \chi_{_{\{1\}}}$ is a Hausdorff log-moment sequence for some real number $c \in [0, w_1].$ Since the sequence $\chi_{\{1\}}$ is also a Hausdorff log-moment sequence with representing measure $\delta_0$ (see Example~\ref{Exm-DMS-1}(c)), the desired conclusion is now immediate from Proposition~\ref{prop-DMS-properties}.  
		\end{proof}

	\section{Proof of Theorem~\ref{thm-char-DMS}}

 We need a couple of facts to complete the proof of Theorem~\ref{thm-char-DMS}. 

\begin{lemma} \label{lem-char-DMS}
For a positive integer $j,$ let $\{w_{n}\}_{n \geqslant j}$ be a Hausdorff log-moment sequence with a representing measure $\mu$ concentrated on $[0, 1].$ If $\mu(\{0\})=0,$ then 
there exists a unique completely monotone function $f : [\log(j),\infty) \rar [0, \infty)$ such that 
\beq \label{1.2}
w_n = f(\log(n)),    \quad n \geqslant j.
\eeq
If this happens, then the representing measure of $f$ is equal to $\varphi_*\mu$ $($see \eqref{phi-psi}$).$
\end{lemma}
\begin{proof} 
Assume that $\mu(\{0\})=0.$ By Lemma~\ref{CoV}(i) (with $\varphi$ given by \eqref{phi-psi}), we obtain 
\beqn w_{n} =  \int_{[0, \infty)} e^{-\log(n) s}\varphi_*\mu(ds), \quad n\geqslant j.
\eeqn
Since $\mu|_{(0, 1]}$ is a Radon measure, 
$\varphi_*\mu$ is a Radon measure concentrated on $[0, \infty)$ (see \cite[Proposition~2.1.15]{BCR}). 
One may now define the function $f : [\log(j),\infty) \rar [0, \infty)$ by 
\begin{equation*} 
f(\lambda) = \int_{[0, \infty)} e^{-\lambda s}\varphi_*\mu(ds), \quad \lambda \in [\log(j), \infty).
\end{equation*}
Thus the representing measure of $f$ is equal to $\varphi_*\mu.$
Note that $f(\lambda) \Le w_{j}$ 
and hence $f$ is well-defined. Clearly, $f(\log(n)) = w_{n}$ for every $n \geqslant j.$ By \cite[Theorem~IV.12b]{W} and the remark prior to \cite[Theorem~IV.12c]{W}], $f$ is completely monotone. This completes the proof of the existence of $f.$ 

To see the uniqueness of $f,$ for $k=1, 2,$ consider the completely monotone function  $f_k : [\log(j),\infty) \rar [0, \infty)$ satisfying $f_k(\log(n))=w_n$ for every integer $n \geqslant j.$ Fix $k=1, 2$ and let $\nu_k$ be a positive Radon measure satisfying 
\begin{equation*} 
f_k(\lambda) = \int_{[0, \infty)} e^{-\lambda s}\nu_k(ds), \quad \lambda \in [\log(j), \infty).
\end{equation*}
For a real number $a,$ let $\mathbb H_a$ denote the right half plane $\{z \in \mathbb C : \Re(z) > a\}.$ 
We now define a
function $F_k$ on the closed right half plane $\overline{\mathbb{H}}_{\log(j)}$ by setting
\beqn 
F_k(z)=\int_{[0, \infty)} e^{-zt}\nu_k(dt), \quad z \in \overline{\mathbb H}_{\log(j)}.
\eeqn 
Since $|e^{-z t}|=e^{-\Re(z) t} \leq e^{-\log(j) t}$ for all $z \in \overline{\mathbb H}_{\log(j)}$ and $t \in [0, \infty),$ by the dominated convergence theorem,
$F_k$ is continuous  on $\overline{\mathbb H}_{\log(j)}.$ By theorems of Fubini and Morera,
it is easily seen that $F_k$ is holomorphic in $\mathbb H_{\log(j)}$ (cf. \cite[Proposition~4.1]{AC}).
Let $L$ be a holomorphic branch of $\log$ defined on the right half plane $\mathbb H_j$ and define $H_k : \mathbb H_0 \rar \mathbb C$ by 
\beqn
H_k(z)= F_k (L(z + j)), \quad z \in \mathbb H_0.
\eeqn
Note that $H_k$ is a well-defined holomorphic function. Since $F_k$ is bounded, so is $H_k.$ 
Since $f_k(\log(n))=w_n$ for every integer $n \geqslant j,$ 
\begin{equation*}
H_k(n)=F_k(L (n+j)) =f_k (\log(n+j))  = w_{n+j}, \quad n \geqslant 0.
\end{equation*}
Clearly, $H_1-H_2$ is a bounded holomorphic function on $\mathbb H_0$ satisfying 
\beqn
\text{$(H_1-H_2)(n)=0$ for every integer $n \geqslant 0.$}
\eeqn 
Hence, by \cite[Corollary~9.5.4]{Bo}, $H_1-H_2$ is identically zero (cf. \cite[Theorem~9.2.1]{Bo}). It follows that  $f_1 (\log(x+j)) =f_2 (\log(x+j))$ on $[0, \infty).$ Since $x \mapsto \log(x+j)$ is a bijection from $[0, \infty)$ onto $[\log(j), \infty),$ $f_1=f_2.$  This completes the proof of the uniqueness part.
\end{proof}
We also need a converse of Lemma~\ref{lem-char-DMS}.
\begin{lemma} \label{rmk-fact}
For a positive integer $j,$ let $\{w_{n}\}_{n \geqslant j}$ be a sequence of non-negative real numbers.  
If there exists a completely monotone function $f : [\log(j),\infty) \rar [0, \infty)$ such that \eqref{1.2} holds, then $\{w_{n}\}_{n \geqslant j}$ is a Hausdorff log-moment sequence and $\{w_{n+j}\}_{n \geqslant 0}$ is a minimal completely monotone sequence.
\end{lemma}
\begin{proof}
Let $f : [\log(j),\infty) \rar [0, \infty)$ be a completely monotone function with the representing measure $\nu$ concentrated on $[0, \infty)$ and assume that \eqref{1.2} holds. By Lemma~\ref{CoV}(ii) (with $\psi$ given by \eqref{phi-psi}), we obtain 
\beqn f(\lambda) =  \int_{(0, 1]} t^{\lambda} \psi_*\nu(dt), \quad \lambda \in [\log(j),\infty).
\eeqn
 It is clear from \eqref{1.2} that $\{w_{n}\}_{n \geqslant j}$ is a Hausdorff log-moment sequence.

To prove the remaining part, 
define $g : [0, \infty) \rar [0, \infty)$ by $$g(x)=\log(x+j), \quad x \in [0, \infty),$$ and note that $g$ is a Bernstein function. Hence, by \cite[Theorem~3.6]{SSV}, ${f} \circ g : [0, \infty) \rar [0, \infty)$ is a completely monotone function. It now follows from \cite[Theorem~IV.14b]{W} that the sequence $\{f \circ g(n)\}_{n \geqslant 0}$ is a minimal completely monotone sequence. However, $f \circ g(n) = w_{n+j}$ for $n \geqslant 0$ completing the verification.
\end{proof}

\begin{proof}[Proof of Theorem~\ref{thm-char-DMS}]
(i) Assume that $\mathbf w = \{w_{n}\}_{n \geqslant 1}$ is a Hausdorff log-moment sequence with the representing measure $\mu.$ 
By Corollary \ref{minimal-coro}(i), the sequence $\mathbf w -  \mu(\{0\}) \, \chi_{_{\{1\}}}$ is a minimal Hausdorff log-moment sequence with a representing measure $\widehat{\mu}$ (see \eqref{mu-hat-def}). 
By Lemma~\ref{lem-char-DMS} (with $j=1$), there exists a unique completely monotone function $f : [0,\infty) \rar [0, \infty)$ such that 
\begin{equation*} 
w_n  - \mu(\{0\})\, \chi_{_{\{1\}}}(n) = f(\log(n)),    \quad n \geqslant 1.
\end{equation*}
Since $w_1 - f(0) =\mu(\{0\}),$ we have $w_1 \geqslant f(0),$ and
we obtain the necessity part. To see the sufficiency part, note that by \eqref{DMS-formula-1} and Lemma~\ref{rmk-fact}, the sequence $\mathbf w -  (w_1-f(0)) \, \chi_{_{\{1\}}}$ is a Hausdorff log-moment sequence. One may now apply Corollary \ref{minimal-coro}(ii).  

(ii)  
This follows from Remark~\ref{minimal-rmk} and Lemmas~\ref{lem-char-DMS} and \ref{rmk-fact}. 

The remaining part follows from Lemma~\ref{lem-char-DMS}. 
\end{proof}

If $\mathbf w$ is a Hausdorff log-moment sequence with the completely monotone function $f$ satisfying either \eqref{DMS-formula-1} ($j=1$) or \eqref{DMS-formula-2} ($j \Ge 2$), then we refer to $(\mathbf w, f)$ as the {\it Dirichlet pair}. 

\begin{example}[Example~\ref{Exm-DMS-1} continued] \label{Exa-DMS-2}
Let us see some examples of Dirichlet pairs $(\mathbf w, f).$
\begin{itemize} 
\item[(a)] For $\alpha < 0,$ consider a function $f_\alpha : [\log(2),\infty) \longrightarrow[0,\infty)$ given by $f(\lambda) = \lambda^{\alpha},$  $\lambda \in [\log(2), \infty).$ Then $f_\alpha$ is a  completely monotone function with the representing measure $\nu_\alpha(dt) =\frac{t^{-1-\alpha}}{\Gamma(-\alpha)}(dt)$, that is,
\begin{equation*}
f_\alpha(\lambda) =	\int_{[0,\infty)} e^{-\lambda t} \nu_\alpha(dt), \quad \lambda \in [\log(2), \infty).
\end{equation*}
If $\mathbf w_\alpha = \{(\log(n))^{\alpha}\}_{n \geqslant 2},$ then $(\mathbf w_\alpha, f_\alpha)$ is a Dirichlet pair.
\item[(b)] For $\alpha > 0,$ consider a function $f_{\alpha}: [0,\infty) \longrightarrow [0,\infty)$ given by $f_\alpha(\lambda) = \frac{1}{\lambda+\alpha},$  $\lambda \in [0, \infty).$ Then $f_\alpha$ is a completely monotone function with the representing measure $\nu_\alpha(dt) = e^{-t \alpha}(dt)$, that is,
\begin{equation*}
f_\alpha(\lambda) =	\int_{[0,\infty)} e^{-\lambda t}\nu_\alpha(dt), \quad \lambda \in [0, \infty).
\end{equation*}
Note that $w_1=f(0).$ Thus if $\mathbf w 
= \big\{\frac{1}{\log(n) + \alpha}\big\}_{n \geqslant 1},$ 
 then $(\mathbf w_\alpha, f_\alpha)$ is a Dirichlet pair.
\item[(c)] For $\alpha \in (0,1],$ consider a function $f_\alpha: [0,\infty) \longrightarrow [0,\infty)$ given by $f_\alpha(\lambda) = \alpha^{\lambda},$ $\lambda \in [0, \infty).$ Then $f_\alpha$ is a  completely monotone function with the representing measure $\nu_\alpha (dt)= \delta_{\{-\log(\alpha)\}}(dt)$, that is,
\begin{equation*}
f_\alpha(\lambda) =	\int_{[0,\infty)} e^{-\lambda t}\nu_\alpha(dt), \quad  \lambda \in [0, \infty).
\end{equation*}
Note that $w_1=f(0).$ Thus, if $\mathbf w_\alpha = \{\alpha^{\log(n)}\}_{n \geqslant 1},$ then $(\mathbf w_\alpha, f_\alpha)$ is a Dirichlet pair.
 Moreover, $(\chi_{_{\{1\}}}, 0)$ is also a Dirichlet pair $($the case of $\alpha =0).$
 \hfill $\diamondsuit$
\end{itemize}
\end{example}

Let us see an application of Hausdorff log-moment sequences to the theory of Helson matrices. Following \cite{PP}, we say that a matrix $\big(a_{m, n}\big)_{m, n = 1}^{\infty}$
is a {\it Helson matrix} if there exists a sequence $\mathbf w = \{w_n\}_{n=1}^{\infty}$ such that 
\beqn
a_{m, n} = w_{mn}, \quad m, n \Ge 1.
\eeqn
In this case, the matrix $\big(a_{m, n}\big)_{m, n = 1}^{\infty}$ is denoted by $M(\mathbf w).$
By \cite[Theorem~5.1]{PP} together with the discussion prior to it, the Helson matrix $M(\mathbf w)$ defines a bounded linear operator on $\ell^2(\mathbb Z_+)$  provided $\mathbf w \in \ell^2(\mathbb Z_+)$ satisfies
\beqn
w_n \Le \frac{C}{\sqrt{n} \log(n)}, \quad n \Ge 2.
\eeqn
This combined with Theorem~\ref{thm-char-DMS}(i) yields the following:
\begin{corollary} \label{coro-Helson}
Assume that $j=1$ and let $(\mathbf w, f)$ be a Dirichlet pair such that $\mathbf w \in \ell^2(\mathbb Z_+)$ and $w_1=f(0).$ If there exists a positive real number $C$ such that
\beqn
xf(x) \Le C e^{-x/2}, \quad x \in [0, \infty),
\eeqn then the Helson matrix $M(\mathbf w)$ defines a bounded linear operator on $\ell^2(\mathbb Z_+).$
\end{corollary}  
The forgoing corollary can be used to check that the Helson matrix $M = \left[\alpha^{\log(mn)}\right]_{m,n = 1}^{\infty}$ 
defines a bounded linear operator on $\ell^2(\mathbb Z_+)$ for any $\alpha \in (0,\frac{1}{\sqrt{e}}).$  
In view of Corollary~\ref{coro-Helson}, this is immediate from Example~\ref{Exa-DMS-2}(c) and the fact that
the function $g(\lambda) = \lambda \alpha^{\lambda} e^{\frac{\lambda}{2}},$ $\lambda \in [0,\infty)$ is a bounded function.

As another application of Theorem~\ref{thm-char-DMS}, we decipher the relation between Hausdorff log-moment sequences and completely monotone sequences.
\begin{corollary}  \label{coro-CM}
For a positive integer $j,$ let $\mathbf w =\{w_{n}\}_{n \geqslant j}$ be a Hausdorff log-moment sequence with a representing measure $\mu.$ Then $\{w_{n+j}\}_{n \geqslant 0}$ is a completely monotone sequence. Moreover, the following statements are valid:
\begin{enumerate}
\item[(i)] if $j=1,$ then   
the representing measure $\nu$ of the completely monotone sequence $\{w_{n+1}\}_{n \geqslant 0}$ is given by
\beq \label{rep-mre-nu}
	\nu(d\sigma)		= \int_{[0,\infty)} \int_{\psi^{-1}(\sigma \backslash \{0\})}  \!\! \left(\frac{\lambda^{s-1} e^{-t}}{\Gamma(s)}\right)dt ~\varphi_* \widehat{\mu}(ds)+\mu(\{0\}) \delta_{0}(d\sigma)
		\eeq 
	for every Borel subset $\sigma$ of $[0, 1],$ where $\varphi,$ $\psi$ are as given in Lemma~\ref{CoV}, and $\widehat{\mu}$ is given by \eqref{mu-hat-def},
	\item[(ii)] if $j \geqslant 2,$ then $\{w_{n+j}\}_{n \geqslant 0}$ is a minimal completely monotone sequence.
		\end{enumerate}
\end{corollary}
\begin{proof}
 Assume that $j=1.$ By Theorem~\ref{thm-char-DMS}(i), 
 there exists a completely monotone function $f : [0,\infty) \rar [0, \infty)$ such that 
\beq  \label{perturbed-moment}
w_n -\mu(\{0\})\, \chi_{_{\{1\}}}(n) = f(\log(n)),    \quad n \geqslant 1.
\eeq
This combined with Lemma~\ref{rmk-fact} yields that $\{w_{n+1} - \mu(\{0\}) \chi_{\{1\}}(n+1)\}_{n \geqslant 0}$ is completely monotone.  Since the sequence $\{\chi_{\{1\}}(n+1)\}_{n \Ge 0}$ is also a completely monotone sequence, by the fact that the set of completely monotone sequences is a cone (see \cite[p.~130]{BCR}), $\{w_{n+1}\}_{n \geqslant 0}$ is completely monotone.   
To see the remaining part, note that by 
Corollary~\ref{minimal-coro}(i), the Hausdorff log-moment sequence $\mathbf w - \mu(\{0\})\, \chi_{_{\{1\}}}$ is a minimal Hausdorff log-moment sequence with the representing measure $\widehat{\mu}$ (see \eqref{mu-hat-def}).
Hence, by Lemma~\ref{lem-char-DMS}, the representing measure of $f$ is given by $\varphi_* \widehat{\mu}$ (see \eqref{phi-psi}). 
We now apply \cite[Theorem~7.2]{SS} to $f$ and the Bernstein function $g(x)=\log(x+1),$ $x \in [0, \infty),$ to conclude that the representing measure $\eta$ of the completely monotone function $f \circ g$ is given by  
\beq \label{zeta-mre}
\eta(\sigma) = \int_{[0, \infty)} \nu_s(\sigma) \varphi_* \widehat{\mu}(ds)~\mbox{for every Borel subset}~\sigma~\mbox{of~}[0, \infty),
\eeq
where $\nu_s$ (the so-called {\it convolution semi-group} of $g$) is the representing measure of the completely monotone function $e^{-s g},$ $s \Ge 0.$ Since $e^{-s g(x)} = \frac{1}{(x+1)^{s}}$ for $x \geqslant 0,$ the measure $\nu_{s}$ is the weighted Lebesgue measure given by
		\beq \label{semig}
		\nu_{s}(dt) = \begin{cases} \frac{t^{s-1} e^{-t}}{\Gamma(s)}dt & \mbox{if~} s \in (0, \infty), \\
		\delta_0 & \mbox{if~}s=0.
		\end{cases}
		\eeq
By \eqref{perturbed-moment} and Lemma~\ref{CoV}(ii), the representing measure of the completely monotone sequence $\{w_{n+1} -\mu(\{0\})\chi_{\{1\}}(n+1)\}_{n \geqslant 0}$ is $\widehat{\psi_* \eta}$ (see \eqref{mu-hat-def}). 
By Corollary~\ref{minimal-coro}, the representing measure of $\{w_{n+1}\}_{n \geqslant 0}$ is $\widehat{\psi_* \eta} + \mu(\{0\})\delta_0.$
This combined with \eqref{zeta-mre} and \eqref{semig} yields \eqref{rep-mre-nu}.

 If $j \geqslant 2,$ then by Theorem~\ref{thm-char-DMS}(ii) and Lemma~\ref{rmk-fact},  $\{w_{n+j}\}_{n \geqslant 0}$ is a minimal completely monotone sequence. This yields (ii).
\end{proof}

\begin{example} For a positive integer $j,$ let $\mathcal{CM}_j$ and $\mathcal{DM}_j$ denote the cones of completely monotone sequences $\{a_{n+j}\}_{n \Ge 0}$ and Hausdorff log-moment sequences $\{a_{n}\}_{n \Ge j}$  respectively. By Corollary~\ref{coro-CM}, $\mathcal{DM}_j$ $\subseteq$ $\mathcal{CM}_j$ for every positive integer $j.$ Moreover, this inclusion is strict: $$\mathcal{DM}_j \subsetneq \mathcal{CM}_j, \quad j \Ge 1.$$
Indeed, for any integer $j \Ge 1,$ 
 $\big\{\frac{1}{n+j+1}\big\}_{n \geqslant 0}$ is a minimal completely monotone sequence with the representing measure $x^{j}dx.$ On the other hand, since the matrix $A = \big(\frac{1}{pq+1}\big)_{p,q = j}^{j+1}$ has determinant less than $0,$ 
by Proposition~\ref{prop-DMS-properties-1}(i), $\big\{\frac{1}{n+1}\big\}_{n \geqslant j}$ is not a Hausdorff log-moment sequence. \hfill $\diamondsuit$
\end{example}

We have already seen that Hausdorff log-moment sequence $\{w_{n}\}_{n \Ge 1}$ is determinate (see Proposition~\ref{unique}). Surprisingly, for any integer $j \Ge 2,$ a Hausdorff log-moment sequence $\{w_{n}\}_{n \Ge j}$ is almost determinate in the following sense:
\begin{corollary} \label{determinate}
 	Let $j \Ge 2$ be a positive integer. If $\{w_{n}\}_{n \geq j}$ is a Hausdorff log-moment sequence, then the restriction of its representing measure to $(0, 1]$ is uniquely determined. 
 \end{corollary}
 \begin{proof}
 	Let $\{w_{n}\}_{n \geq j}$ be a Hausdorff log-moment sequence with two representing measures $\mu$ and $\nu.$ 
 By Theorem~\ref{thm-char-DMS}, there exists a unique completely monotone function $f : [\log(j),\infty) \longrightarrow \mathbb{R}$ with representing measure $\varphi_*\mu$ (see \eqref{phi-psi}) such that $f(\log(n)) = w_{n}$ for every integer $n \geqslant j.$
Note that representing measure of a completely monotone function $f : [\log(j), \infty) \rar [0, \infty)$ is unique. This fact may be derived from Bernstein's Theorem (see \cite[Theorem~1.4]{SSV}) by considering the completely function on $[0, \infty)$ given by $s \mapsto f(s+\log(j)),$ $s \in [0, \infty).$ We
may now conclude that $\varphi_*\mu$ and $\varphi_*\nu$ must agree on $[0, \infty).$ Since $\varphi : (0, 1] \rar [0, \infty)$ is a bijection, we obtain $\mu|_{_{(0, 1]}} = \nu|_{_{(0, 1]}}$ completing the proof. 
\end{proof}
 

We conclude the paper with some remarks revealing the relationship between Hausdorff log-moment sequences $\mathbf w=\{w_n\}_{n \Ge 1}$ and the associated linear functional $L_{\mathbf w}.$ 
Note that the conclusion of Proposition~\ref{prop-DMS-properties-1} can be rephrased in terms of the functional $L_{\mathbf w}$ to derive the following necessary conditions for a sequence $\mathbf w$ to be a Hausdorff log-moment sequence.
\begin{proposition} \label{fact-end}
If $\{w_n\}_{n \Ge 1}$ is a Hausdorff log-moment sequence, then 
\beq \label{Cond-1}
L_{\mathbf w}(k^{-s}q^2) &\Ge 0,& \quad q \in \mathcal D[s], ~k \Ge 1.
\eeq
\end{proposition}
\begin{proof} 
To see \eqref{Cond-1}, note that
by Proposition~\ref{prop-DMS-properties-1}(iii), for every integer $k\geqslant 1,$ $\{w_{kn}\}_{n \geqslant 1}$ is a Hausdorff log-moment sequence. Hence, by Proposition~\ref{prop-DMS-properties-1}(i),
$(w_{kpq})_{p, q \in F}$ is positive semi-definite for every finite subset $F$ of $\mathbb Z_+.$ This is easily seen to be equivalent to \eqref{Cond-1} completing the proof.
\end{proof}
We do not know whether the conditions \eqref{Cond-1}  ensure that the sequence $\mathbf w$ is a Hausdorff log-moment sequence (cf. \cite[Theorem~3.12]{Sc} and \cite[Theorem~3.3]{PP}).  In particular, one may ask for a counterpart of \cite[Proposition 3.2]{Sc} for the Dirichlet polynomials (cf. the discussion following \cite[Theorem~3.3]{PP}). In the unavailability of a required "Positivstellensatz'', it seems desirable to find a counterpart of the method employed in \cite[Theorem~2.3]{Va} for Dirichlet polynomials. We believe these questions warrant
 additional attention.

\section*{Appendix: Second proof of Theorem~\ref{Q1-DS-Coro}}


It seems that there could be a direct proof of Theorem~\ref{Q1-DS-Coro}, which resembles a part of the solution of Hausdorff moment problem  (see \cite[Proposition~4.17.7]{Si-1}). 
Interestingly, there is an alternate proof of Theorem~\ref{Q1-DS-Coro} of topological flavour and the idea of the proof has been known in the literature (see, for instance, \cite[Proofs of Theorems~9.15 and 9.19]{Sc} and \cite[Proof of Proposition~2.1]{PS}). 

\begin{proof}[Second proof of Theorem~\ref{Q1-DS-Coro}]
We only verify the implication (i)$\Rightarrow$(ii). In view of \cite[Proposition~1.9]{Sc}, we may assume that $K$ is unbounded. Assume that $L_{\mathbf w}$ is $\mathcal D_{+}[s]$-positive. 
Let $[0, \infty]$ denote the Alexandroff one-point compactification of $[0, \infty)$ (which is always Hausdorff) and let $C([0, \infty], \mathbb R)$ denote the space of real-valued continuous functions on $[0, \infty].$ 
Let $\mathcal S$ denote the real linear space of functions $F : [0, \infty] \rar \mathbb R$ given by 
\beqn
\mathcal S= \{F \in C([0, \infty], \mathbb R) ~|~ \text{there exists a}~f \in \mathcal D[s]~\text{such that}~F|_{[0, \infty)} = f\}.
\eeqn
Note that for every $F \in \mathcal S,$ there exists a unique $f \in \mathcal D[s]$ such that $F|_{[0, \infty)} = f$ and $F(\infty) = \lim_{s \rar \infty}f(s).$
Moreover, if $f \in \mathcal D[s],$ then $F$ given by
\beqn
F(s) = \begin{cases}  f(s) & \mbox{if~} s \in [0, \infty), \\
\lim_{s \rar \infty}f(s) & \mbox{if~} s  = \infty,
\end{cases}
\eeqn
defines an element in $\mathcal S.$

Define a linear functional $M_{\mathbf w} : \mathcal S \rar \mathbb R$ by $M_{\mathbf w}(F) = L_{\mathbf w}(f).$ Let $\overline{K}$ denote the closure of $K$ in $[0, \infty].$ 
By assumption, $M_{\mathbf w}$ is $\mathcal S^{\overline{K}}_+[s]$-positive, and hence by  \cite[Proposition~1.9]{Sc}, there exists a positive Radon measure $\nu$ on $\overline{K}$ such that
\beqn 
M_{\mathbf w}(F) = \int_{\overline{K}} F(s) \nu(ds), \quad F \in \mathcal D^{\infty}[s].
\eeqn	
It follows from the discussion in the previous paragraph that for any $f \in \mathcal D[s],$ there exists $F \in \mathcal S$ such that 
\beqn
L_{\mathbf w}(f) &=& \int_{\overline{K}} F(s) \nu(ds) \\
&=& \nu(\{\infty\})\, F(\infty) + \int_{K} f(s) \nu(ds) \\
&=& (w_1 - \nu(K))\lim_{s \rar \infty}f(s) + \int_{K} f(s) \nu(ds),
\eeqn
where we used the equality that $w_1 = \nu(\overline{K})$ $($which can be seen by arguing as in the proof of Theorem~\ref{Q1-DS-Coro}$)$ and fact that $\overline{K} = K \sqcup \{\infty\}$ (disjoint union).
This completes the proof once we notice that $\nu$ is a finite measure.
\end{proof}
\end{document}